\documentclass[10pt]{amsart}

\usepackage[english]{babel} 
\usepackage[utf8]{inputenc}

\usepackage[shortlabels]{enumitem}
\usepackage{amsfonts}
\usepackage{amsmath}
\usepackage{amssymb}
\usepackage{amsthm}
\usepackage{mathtools}
\usepackage{stmaryrd}
\usepackage{hyperref}
\usepackage{mathrsfs}
\hypersetup{
    colorlinks=false,
    linkcolor=blue
}
\usepackage[capitalize]{cleveref}
\usepackage{tikz}
\usetikzlibrary{cd}

\usepackage{graphicx}
\usepackage{algorithm}
\usepackage{algpseudocode}
\usepackage{booktabs}

\usepackage{colonequals}

\newtheorem{thm}{Theorem}[section]

\newtheorem{prop}[thm]{Proposition}

\newtheorem{quest}[thm]{Question}

\theoremstyle{definition}
\newtheorem{defi}[thm]{Definition}

\newtheorem{exam}{Example}[section]

\theoremstyle{remark}
\newtheorem{rem}[thm]{Remark}

\newcommand{\N}{\mathbb N}
\newcommand{\Z}{\mathbb Z}
\newcommand{\Q}{\mathbb Q}

\newcommand{\C}{\mathbb C}

\newcommand{\Proj}{\operatorname{Proj}}
\newcommand{\Spec}{\operatorname{Spec}}

\newcommand{\FF}{\mathbb{F}}
\newcommand{\PP}{\mathbf{P}}

\renewcommand{\mod}{\operatorname{mod }}

\renewcommand{\phi}{\varphi}

\newcommand{\Homog}{\operatorname{Homog}}
\newcommand{\rig}{\operatorname{rig}}
\newcommand{\Frob}{\operatorname{Frob}}
\newcommand{\dR}{\operatorname{dR}}

\newcommand{\newt}{\operatorname{newt}}

\newcommand{\ins}{\subseteq}
\newcommand{\isom}{\cong}
\newcommand{\hodge}{\operatorname{hodge}}

\usepackage[OT2,T1]{fontenc}
\DeclareSymbolFont{cyrletters}{OT2}{wncyr}{m}{n}
\DeclareMathSymbol{\Sha}{\mathalpha}{cyrletters}{"58}

\nocite{*}

\title[Newton strata realization for hypersurfaces]{Newton strata realization for hypersurfaces via
explicit \(p\)-adic cohomology}

\author[Batubara]{Ryan Batubara}
\address{Ryan Batubara, Department of Mathematics, University of California San Diego, 2985 Muir Ln, La Jolla, CA 92093, USA}
\email{rbatubara@ucsd.edu}
\date{\today}

\author[Garzella]{Jack J Garzella}
\address{Jack J Garzella, Department of Mathematics, University of California San Diego, 2985 Muir Ln, La Jolla, CA 92093, USA}
\email{jgarzell@ucsd.edu}
\date{\today}

\author[Huang]{Yongyuan Huang}
\address{Yongyuan Huang, Department of Mathematics, University of California San Diego, 2985 Muir Ln, La Jolla, CA 92093, USA}
\email{yoh011@ucsd.edu}
\date{\today}

\author[Mellberg]{Maximus Mellberg}
\address{Maximus Mellberg, Department of Mathematics, University of California San Diego, 2985 Muir Ln, La Jolla, CA 92093, USA}
\email{mmellber@ucsd.edu}
\date{\today}

\begin{document}

\begin{abstract}
    Let \(X\) be a smooth projective hypersurface over a finite field \(k\) of characteristic \(p\).
    We address the problem of practically computing the zeta function \(Z(X,T)\) of \(X\)
    (equivalently, the point counts \(\#X(\mathbb{F}_{q})\), where \(q = p^{n}\)),
    and we focus on the case when \(7 \leq p < 50\).
    We use the theoretical framework of the variant of Kedlaya's algorithm
    in \cite{AKR}, 
    and we use the technique of \textit{controlled reduction} as described
    in \cite{costa}.
    We define an optimization problem that abstracts the 
    key bottleneck in the implementation
    of controlled reduction.
    An algorithm that solves this problem is called a \textit{reduction policy}.
    We present three reduction policies with different advantages 
    and disadvantages.
    We also present a high-performance implementation of controlled reduction
    that contains GPU-optimized linear algebra code and
    a data structure for linear recurrences that the authors hope can
    be used to study further reduction policies.
    Our algorithms get state-of-the-art performance in many cases;
    for example, we beat \cite{Tuitman} or \cite{Kyng} on many examples of quintic curves,
    while also being able to compute zeta functions of cubic fourfolds when \(p = 7\).
    We also have the first (to our knowledge) systematic computations of zeta functions
    of quintic surfaces.
    We use our implementation to deduce many new explicit examples of 
    varieties with specified Newton polygons,
    including a cubic fourfold which are neither ordinary
    nor supersingular,
    quartic K3 surfaces of various Artin-Mazur heights,
    and quintic surfaces of all possible domino numbers.
\end{abstract}

\maketitle

\section{Introduction}
\subsection{Background}
Let $k=\FF_q$ and $X/k$ be a smooth proper $k$-scheme of finite type of dimension $n$. The \textit{zeta function} of $X$ is the generating series of the number of $\FF_{q^r}$--points on $X$ defined by 
\[
Z(X,T)\coloneq\exp\left(\sum_{r=1}^{\infty}\frac{T^r}{r}\#X(\FF_{q^r})\right);
\] 
it is an invariant of $X$ of substantial interest in number theory and arithmetic geometry. The well-celebrated Weil conjectures \cite{Weil} proven by Deligne \cite{WeilI}\cite{WeilII} encapsulate many key properties of $Z(X,T)$. For example, $Z(X,T)$ is a rational function of the form
\(Z(X,T) = \prod_{i=0}^{2n} P_{i}(x)^{(-1)^{i+1}}\) for \(P_{i} \in \mathbb{Z}[T]\). Let \(P_{i} = \sum_{j} a_{ij}x^{j}\). We define the \textit{\(i\)-th Newton polygon} \(\newt_{i}(X)\) of \(X\) to be the lower convex hull of the points \((j, v_{p}(a_{ij}))\); it is another invariant of \(X\) which is coarser than its zeta function but also of wide interest. In many cases (such as when \(X\) is a curve, abelian variety, or projective hypersurface), there is only one \(i\) for which the factor \(P_{i}\) is nontrivial, in which case we say that \(\newt_{i}(X)\) is \textit{the} Newton polygon of \(X\) and denote it by \(\newt(X)\). A celebrated theorem of Mazur \cite{Mazur} gives a lower bound for \(\newt_{i}(X)\) in terms of another lower convex hull, which is called the \textit{Hodge polygon} and is denoted \(\hodge(X)\).



\subsection{Motivation}
We recall some of the reasons for interest in \(\newt(X)\).

\subsubsection{Other Invariants of \(X\)}

Many other arithmetic invariants of \(X\) are encoded in \(\newt(X)\).
If \(X\) is a curve of genus $g$ or abelian variety of dimension $g$,
the \textit{\(p\)-rank} of \(X\) is 
exactly the length of the slope zero segment of \(\newt(X)\).
Furthermore, \(X\) is called \textit{ordinary} if its Newton polygon
equals its Hodge polygon, and
\textit{almost ordinary} if it has the (in this case, unique)
lowest Newton polygon that is not ordinary.
In fact, being ordinary is equivalent to having \(p\)-rank \(g\),
and being almost ordinary is equivalent to having \(p\)-rank \(g-1\).
If \(X\) is a K3 surface, then by Mazur's theorem and the 
Weil conjectures, the first slope of the Newton polygon
of \(X\) must be equal to \(1 - \frac{1}{h}\), for \(h \in \{1\ldots10\} \cup \infty\).
This \(h\) is called the \textit{Artin-Mazur height} of \(X\) \cite{Artin-Mazur}.
More generally, if \(X\) is any surface, we may characterize its
\textit{domino number} \cite{Joshi-Rajan, Joshi} as the smallest vertical distance between
its Newton polygon and its Hodge polygon. 
We say that \(X\) is \textit{Hodge-Witt} if its domino number is zero. 
For example, a K3 surface is not Hodge-Witt if and only if it has height \(\infty\).
Note that many of the aforementioned notions first appeared with other definitions,
but the Newton polygon provides a clean way to discuss all of them in a unified way.

These constructions can be extended to other families of varieties.
To define an analogue of the \(p\)-rank for \(X\) a cubic threefold,
we must consider the length of the slope one segment of
\(\newt(X)\); the \(p\)-rank of a cubic threefold coincides with the \(p\)-rank of
its intermediate Jacobian \cite{int-jacobian}.
For \(X\) a cubic fourfold, we may define a ``height'' of \(X\) such
that the first slope of the Newton polygon is \(2 - \frac{1}{h}\),
where \(h \in \{1\ldots10\} \cup \infty\).
This height probably deserves to be called the ``Artin-Mazur height'' or
``higher Artin-Mazur height'' of the cubic fourfold; however,
due to the lack of a well-developed theory, we refrain from naming it. Instead, we opt
simply to use the term \textit{Newton height} instead.
One of the key motivations for this work is to lay the groundwork for the development of the theory of
heights of cubic fourfolds.

\subsubsection{\(F\)-singualrities and cubic hypersurfaces}
Newton polygons also carry information about singularities. We refer the reader to \cite{ma-polstra-2021-F-sing-comm-alg}
for an overview of the theory of \(F\)-singularities. We say a ring \(R\) of characteristic \(p\) is \textit{\(F\)-split} if the Frobenius map 
\(F \colon R \xrightarrow{} R\) splits in the category of \(R\)-modules.
\(F\)-splitness is a ``mild singularities'' condition; that is,
if \(R\) is regular, then \(R\) is \(F\)-split, while the converse
does not hold \cite[\S1-2]{ma-polstra-2021-F-sing-comm-alg}.
We say that the scheme \(\Spec R\) is \(F\)-split if \(R\) is.
If \(X\) is a smooth projective K3 surface, that is, \(\omega_{X} \isom \mathcal{O}_{X}\),
we have

\begin{prop}
    \label{prop:fsplit:ordinary}
    Choose some \(R\) such that \(X = \Proj R\).
    Let \(h\) be the Artin-Mazur height of \(X\).
    Then the following are equivalent:
    \begin{enumerate}[(i)]
        \item \(X = \Proj R\) is ordinary
        \item \(h = 1\)
        \item \(\Spec R\) is \(F\)-split.
    \end{enumerate}
\end{prop}
We see that the Newton polygon of \(X\) gives information about
the \(F\)-singularities of (any) affine cone over that variety; we may ask for analogues of Proposition \ref{prop:fsplit:ordinary} for other
families of hypersurfaces. 
Recently, Kawakami and Witaszek \cite{Kawakami-Witaszek} have introduced the notion of 
\textit{higher \(F\)-injectivity}, which gives a natural analogue of \(F\)-splitness.
The following is a folklore question, first asked by Witaszek:

\begin{quest}
    \label{quest:higherfinj:ordinary}
    Let \(X\) be a smooth projective cubic hypersurface of dimension \(d \geq 3\),
    and choose \(R\) such that \(X = \Proj R\).
    Let \(h\) be the Newton height of \(X\).
    Are the following equivalent?
    \begin{enumerate}[(i)]
        \item \(X = \Proj R\) is ordinary.
        \item \(h = 1\)
        \item \(\Spec R\) is \(1-F\)-injective in the sense of Kawakami and Witaszek.
    \end{enumerate}
\end{quest}


Motivated by Question \ref{quest:higherfinj:ordinary} 
and the desire to test conjectures with examples, 
we ask the following much simpler question.



\begin{quest}
    \label{quest:strata:realization}
    Can we produce explicit examples of cubic threefolds (resp{.} cubic fourfolds) \(X\) which have 
    interesting Newton polygons? 
    Especially, can we find explicit examples which are ordinary, and examples which have
    Newton polygon that are as close to ordinary as possible without being ordinary?
\end{quest}

Question \ref{quest:strata:realization} may be called the \textit{Newton strata realization problem},
a reference to the fact that Newton polygons stratify families of varieties in characteristic \(p\) 
(see \cite{de-jong-oort-2000-newton-strata}).
The purpose of this paper is to provide a general (computational) methodology for answering questions
such as Question \ref{quest:strata:realization}. See \ref{subsec:fourfolds} for the prior work on this
question for cubic fourfolds.

\subsection{Kedlaya-type algorithms}
One immediate way of accessing the Newton polygons of \(X\) is via explicitly computing \(Z(X,T)\).
Given an arbitrary smooth projective $k$-scheme $X$ of finite type, explicitly computing $Z(X,T)$ remains an open problem in general. 
There has been a myriad of papers written on computing the zeta functions of varieties of particular types, 
most of which use cohomological methods.
For example, the work of Schoof \cite{Schoof} may be interpreted as using mod-\(\ell\) étale
cohomology, the work of Sperber-Voight \cite{Sperber-Voight} uses Dwork cohomology, and
the work of Harvey, Sutherland, Costa, and others 
\cite{harvey-sutherland-2014-hasse-witt-1, 
    harvey-sutherland-2016-hasse-witt-2,
harvey-massierer-sutherland-2016-hyperell-genus-3,
sutherland-2020-superelliptic,
afp-2022-picard-curves,
costa-harvey-sutherland-2023-plane-quartic}
 uses
\v{C}ech cohomology (via formulas for the Hasse-Witt/Cartier-Manin matrix).
The landmark work of \cite{Kedlaya} for hyperelliptic curves 
initiated the study of computations of \(Z(X,T)\)
using rigid-cohomological methods. 
Kedlaya's algorithm has since been generalized and improved in various ways.
We highlight a few relevant developments.
In \cite{AKR}, Abbott, Kedlaya, and Roe generalize Kedlaya's algorithm
to projective hypersurfaces.
Harvey \cite{Harvey} then improves the complexity in \(p\) to 
\(p^{1 / 2}\) using a technique known as \textit{controlled reduction} in the literature.
Costa and Harvey 
\cite{costa-harvey-unpublished-controlledreduction, costa} 
provide algorithms and an implementation for computing \(Z(X,T)\) in the case of projective hypersurfaces; 
Costa, Harvey, and Kedlaya \cite{CR} further generalize these
algorithms to hypersurfaces in toric varieties, under the
condition that the equations of these varieties satisfy
a nondegeneracy condition (see Definition \ref{nondegenerate}).
Such \textit{Kedlaya-type} algorithms remain the state of the art
for the computation of \(Z(X,T)\) when \(\dim X > 1\).

\subsection{Anatomy of a Kedlaya-type algorithm}\label{anatomy}

Let \(H^{i}_{rig}(X)\) denote the rigid cohomology groups of \(X\).
Each Kedlaya-type algorithm works to compute the action of \(\Frob_{q}\)
on \(H^{i}_{rig}(X)\) by the standard method of linear algebra:
write an explicit basis for \(H^{i}_{rig}(X)\), 
write a formula for the action of \(\Frob_{q}\) on an arbitrary element \(\omega\in H^{i}_{rig}(X)\),
 compute the action of \(\Frob_{q}\) on each of the basis vectors, and write \(\Frob_{q}(\omega)\) 
 as a \(p\)-adic linear combination of the basis elements using cohomological relations via some sort of a reduction algorithm. 
 
There are two main difficulties in any Kedlaya-type algorithms:

\begin{enumerate}[(a)]
    \item The formula for the action of \(\Frob_{q}\) on  \(H^{i}_{rig}(X)\)  is a convergent\footnote{
        overconvergent, even
    }
        \(p\)-adic power series with infinitely
        many nonzero terms. 
        One may apply the reduction algorithm to each term
        in the series, but a computer can only do this
        for a finite subset of the terms.
    \item Applying the cohomological relations naively involves a lot of polynomial
        arithmetic, which makes the algorithm impractical in most cases.
\end{enumerate}

Section \ref{precision} summarizes the approach of \cite{AKR} to address the first problem. In regards to the second problem,
Harvey and Costa in \cite{Harvey, costa, costa-harvey-unpublished-controlledreduction} pioneered the method of
\textit{controlled reduction}, which we recall in Section \ref{subsec:cr}.

\subsection{Our contributions}
Let \(X\) be a smooth projective hypersurface. We answer questions akin to Question
\ref{quest:strata:realization} by developing and implementing fast algorithms for calculating \(Z(X,T)\) and deduce \(\newt(X)\) from the zeta function. 
In particular, we use the machinery of \textit{controlled reduction} due to Harvey and Costa.
Previous work on Kedlaya-type algorithms and controlled reduction has focused on the
complexity properties of the algorithms involved per se. 
We take a different perspective and instead focus on practicalities in the implementation.
In particular, the bottleneck of most controlled reduction algorithms is a series of matrix-vector
multiplications which take place over \(\mathbb{Z}/p^{M}\mathbb{Z}\) for some ``precision'' \(M\).
This is also a bottleneck for practical implementations of \cite{Harvey-Trace-Formula}, which is theoretically very
different but algorithmically similar.
Much work has gone into making the size of these matrices as small as possible and to computing
\(Z(X,T)\) for very large primes \(p\) \cite{harvey-sutherland-2014-hasse-witt-1, CR}.
Furthermore, implementing a Kedlaya-type algorithm entails numerous design choices, each involving
trade-offs. Prior implementations often either adopt ``naive'' default options or make selections tailored to the specific type of variety studied in the accompanying work. As a result, such code tend to be difficult to reuse and slow for other types of inputs.

In this work, for \(X\) a projective hypersurface, we describe an algorithm to compute \(Z(X,T)\) that is equivalent to controlled reduction as presented in \cite[Proposition 1.15]{costa}.
However, we give a mathematical formalism for a key implementation detail of controlled reduction
algorithms, which we call the \textit{reduction policy}. 
We describe the choices in Costa's implementations \cite{controlledreduction,
ToricControlledReduction}, and provide two reduction policies which perform better on our examples;
one which is generic, and one which is suited to hypersurfaces of low degree.
We also develop a data structure which encapsulates the functionality necessary for another
important implementation detail, called a \textit{linear recurrence of matrices}.
We provide several implementations, each suited to various \(X\).
Finally, we provide a few GPU-accelerated linear algebra utilities, including
a matrix multiplication modulo \(p^{m}\) for \(p^{m} < 2^{25}\) 
(which is a wrapper around Nvidia's CUBLAS, similar to the \textsc{Magma} implementation),
and a matrix multiplication modulo \(p^{m}\) for \(p^{m} < 2^{53}\),
which uses Karatsuba's multiplication algorithm.



All of our results are obtained on small computers,
each with no more than 32 GB of memory.
Furthermore, given the motivations in \(F\)-singularities, we focus on medium-low characteristic, whereas previous work
have mostly focused on optimizing performance for large primes; 
in this paper, all computations have \(7 \leq p < 50\), though our algorithms often work beyond
these cases.
In particular, we achieve the first systematic computations (to the knowledge of the authors)
on cubic fourfolds in characteristic \(7\) and quintic surfaces in small characteristics;
we achieve state-of-the-art practical performance on cubic threefolds with few monomials, 
as well as state-of-the-art single-core performance on quartic surfaces with few terms.
Furthermore, we have 
performance competitive with \cite{ToricControlledReduction} on random cubic threefolds,
and both dense and sparse cubic surfaces.
Finally, we have performance that is competitive and often beats Tuitman's algorithm \cite{Tuitman}
and/or Kyng's Harvey Trace
Formula algorithm \cite{Kyng} for smooth plane quintics.

With this performance, our main results are as follows. 
\begin{thm}\label{thm:main}
There exist 
\begin{enumerate}
    \item cubic fourfolds with height 2 in characteristic \(p = 7\),
    \item cubic threefolds that are ordinary and almost ordinary in characteristics \(p = 7, 11,13\),
    \item quartic K3 surfaces of height up to 4 in characteristic \(p = 17, 19, 23\),
    \item quintic surfaces of all possible domino numbers in characteristic \(p = 7, 13, 17\),
    \item plane quintics of all possible \(p\)-ranks for \(p = 7, 11, 13\),
    \item plane sextics of all \(p\)-ranks \(\geq 6\) for p = \(7,11,13\), and 
    \item plane heptics which are ordinary and almost ordinary for \(p = 7,11,13,17,19,23,31\).
\end{enumerate}
\end{thm}
Our proof is constructive; in particular, we have explicit examples for all of the above and we have
calculated the full zeta function for each of our examples. All examples can be found in \cite{zenodo}.

Our implementation in the Julia programming language of the algorithms described in this article as well as computations verifying Theorem \ref{thm:main} can be found in
\[\text{\url{https://github.com/UCSD-computational-number-theory/DeRham.jl} and }\]
\[\text{\url{https://github.com/UCSD-computational-number-theory/GPUFiniteFieldMatrices.jl}.}\]

\subsection{Contents}
We conclude this introduction with an outline of the paper. After recalling the theoretical background for computing the zeta function of smooth projective hypersurfaces using explicit \(p\)-adic cohomological method in Section \ref{sec:griffiths}, we dissect the components of controlled reduction in Section \ref{sec:pole-reduction} and formally define an optimization problem that models controlled reduction in Section \ref{subsec:abstract}. We devote Section \ref{sec:policies} to describe three different reduction policies and Section \ref{sec:algo} to present the GPU-accelerated linear algebra algorithms and data structure for computing linear recurrence of matrices that we implement. Section \ref{sec:newton-realization} demonstrates the practical application of our work in finding varieties with Newton polygons previously unknown in the literature, and Section \ref{sec:performance} compares the performance of our algorithms to existing implementations.  

\subsection{Acknowledgements}
The authors thank Edgar Costa and Kiran S. Kedlaya for helpful conversations, and the maintainers and support staff of the UCSD research cluster for use of their machines. The second author was partially supported by the National Science Foundation Graduate Research Fellowship Program under Grant No. 2038238, and a fellowship from the Sloan Foundation.

\section{Computing the Zeta Functions of Projective Hypersurfaces}\label{sec:griffiths}
We recall the necessary theoretical background (see, e.g., \cite[Section 2]{AKR} and the references therein for further details) and establish our notations. 

\subsection{Some $p$-adic cohomology theories}\label{padic}
Suppose $q=p^n$ elements for some prime number $p$. Let $\Q_q$ be the unramified extension of $\Q_p$ of degree $n$, and $\Z_q$ be the ring of integers of $\Q_q$. Let $\mathcal{X}=\PP^n_{k}$, and $\mathcal{Z}$ be a smooth hypersurface of degree $d$ in $\mathcal{X}$ defined by a homogeneous polynomial $f_{\mathcal{Z}}\in k[x_0,\ldots,x_n]_d$. Berthelot \cite{Berthelot} has developed a Weil cohomology theory known as \textit{rigid cohomology}, which contravariantly associates finite dimensional $\Q_q$-vector spaces $H^i_{\rig}(\cdot)$ to a variety. Let $\mathcal{U}=\mathcal{X}\backslash\mathcal{Z}$. For $m\geq 1$, let $\Homog_m$ denote the free $\Z_q$-module of homogeneous polynomials in $\Z_q[x_0,\ldots,x_n]$ of degree $m$; by convention, we let $\Homog_0=\{1\}$ and $\Homog_m=\emptyset$ for $m<0$. After choosing an arbitrary lift $f\in\Homog_d$ of $f_{\mathcal{Z}}$, let $Z$ be the hypersurface cut out by $f$ in $X=\PP^n_{\Z_q}$ and define $U=X\backslash Z$. (For smooth projective hypersurfaces, the existence of such a lift is guaranteed by lifting the coefficients of $f_{\mathcal{Z}}$.)  By a combination of the Lefschetz hyperplane theorem (see e.g. \cite[Theorem 3.1.17]{positivity}), Weil cohomology theories, comparison theorems due to Shiho \cite{Shiho} and Kato \cite{Kato} between rigid cohomology and algebraic de Rham cohomology via crystalline cohomology, calculations from \cite{AKR}, and work of \cite{Griffiths}, letting $Q(T)=\det(1-q^{-1}T\Frob_q|H^n_{\dR}(U_{\Q_q}))$, we always have 
\[Z(\mathcal{Z},T)=\frac{Q(T)^{(-1)^n}}{\prod_{i=0}^{n-1}(1-q^iT)}.\]
Hence it suffices to compute $Q(T)$ in order to compute the zeta functions of smooth projective hypersurfaces. 

\subsection{Griffiths-Dwork Construction}
Griffiths \cite{Griffiths} gives an explicit construction of $H^n_{\operatorname{dR}}(U_{\Q_q})$ as follows. 
Let $\Omega\coloneq\sum_{i=0}^n (-1)^ix_idx_0\wedge\cdots\wedge\widehat{dx_i}\wedge\cdots\wedge dx_n$, where $\widehat{dx_i}$ means the $dx_i$ is omitted. Griffiths shows that $H^n_{\operatorname{dR}}(U_{\Q_q})$ is isomorphic to the $\Q_q$-vector space generated by $\{g\Omega/f^m|m\geq 1, g\in\Homog_{dm-n-1}\}$ modulo the relations
\begin{equation}\label{Griffiths-Dwork}
    \left\{\left(f\frac{\partial g}{\partial x_i}-mg\frac{\partial f}{\partial x_i}\right)\frac{\Omega}{f^{m+1}}:1\leq m\leq n, g\in\Homog_{dm-n}\right\}.
\end{equation} We refer to (\ref{Griffiths-Dwork}) as the \textit{Griffiths-Dwork relations}. Hence every element $\omega\in H^n_{\dR}(U_{\Q_q})$ can be represented by $g\Omega/f^m$ with $m$ minimal and $g\in\Homog_{dm-n-1}$; we refer to such an integer $m$ as the \textit{pole order} of $\omega$.  In fact, the basis elements of pole order \(m\) 
span the \(m\)-th graded piece of the Hodge filtration of \(H^n_{\dR}(U_{\Q_q})\) \cite[\S8]{Griffiths}.
Thus, the maximum pole order in any bases of \(H^n_{\dR}(U_{\Q_q})\) is \(n\).
As \cite[Remark 3.2.5]{AKR} suggests and also with the anatomy of a Kedlaya-type algorithm described in \ref{anatomy} in mind, Griffiths's description of $H^n_{\dR}(U_{\Q_q})$ gives a natural way of computing a $p$-adic approximation of the \textit{Frobenius matrix}, i.e. the matrix representation $F$ of $\Frob_q|H^n_{\dR}(U_{\Q_q})$ with respect to an explicitly computable basis:
\begin{enumerate}
    \item For $m=1,\ldots,n$, descend a basis for $\Homog_{dm-n-1}$ to a basis for $H^n_{\dR}(U_{\Q_q})$.
    \item Compute the action of Frobenius on basis elements of $H^n_{\dR}(U_{\Q_q})$ expressed as truncated power series.
    \item Apply the relations (\ref{Griffiths-Dwork}) to reduce the pole order of terms in the power series obtained from (2). This step is commonly referred as \textit{pole reductions}. 
\end{enumerate}

We assume $q=p$ from now on and denote $\Frob_p$ by $\sigma$. By convention, the natural numbers include \(0\). For \(\beta=(\beta_0,\ldots,\beta_{n})\in\N^{n+1}\), let \(x^{\beta}\coloneq x_0^{\beta_0}\cdots x_n^{\beta_n}\). We also let \(\partial_i f \coloneq \partial f/\partial x_i\).

\subsection{Computing the basis of cohomology}\label{basis}
For the sake of computational efficiency, we prefer having a \textit{monomial} basis for $H^n_{\dR}(U_{\Q_p})$. 
Following Griffiths's construction, computing a monomial basis $B$ for $H^n_{\dR}(U_{\Q_p})$ amounts to computing a monomial basis for the cokernel of the following linear transformation:
\begin{align*}
M_{k}: (\Homog_{k-(d-1)})^{\oplus n+1}&\to \Homog_{k}\\
(\mu_0,\ldots,\mu_n)&\mapsto \sum_{i=0}^{n}\mu_i\partial_i f 
\end{align*}
for $k=dm-n-1$, where $m=1,\ldots,n$.

\subsection{Action of Frobenius}\label{frobenius}
We can express the action of Frobenius on a basis element $x^{\beta}\Omega/f^m$ of $H^n_{\dR}(U_{\Q_p})$ via a series expansion as follows. A calculation using geometric series yields
\begin{equation}\label{frob} 
\sigma\left(\frac{x^{\beta}\Omega}{f^m}\right)=p^n\frac{x^{p(\beta+1)-1}\Omega}{f^{pm}}\sum_{i\geq 0}\binom{-m}{i}\left(\frac{\sigma(f)-f^p}{f^p}\right)^i.
\end{equation}
The number of terms on the RHS of (\ref{frob}) visibly depends on $p$, and this is one of the primary reasons that the run-time of the algorithm from \cite{AKR} is exponential in $p$. Using an insight of Harvey \cite[Proposition 4.1]{Harvey}, one can rewrite (\ref{frob}) in a way such that the number of terms in the series expansion of \(\sigma(x^{\beta}\Omega/f^m)\) does not depend on $p$.  
\begin{prop}
Let $x^{\beta}\Omega/f^m$ be an integral basis element of $H^n_{\dR}(U_{\Q_p})$. Let $N\coloneq N_m$, the series order, and $s=N+m-1$ and $C_{i,\alpha}$ be the coefficient of $x^{\alpha}$ in $f^i$. For $0\leq i<N$, define 
\[D_{i,m}\coloneq \sum_{i=j}^{N-1}(-1)^{i+j}\binom{-m}{i}\binom{i}{j}.\] We have 
\[p^{m-n-1}\sigma\left(\frac{x^{\beta}\Omega}{f^m}\right)\equiv \sum_{j=0}^{N-1}\sum_{|\alpha|=dj}p^{m-1}(D_{j,m}C_{j,\alpha}\operatorname{ mod }p^s)\frac{x^{p(\beta+\alpha+1)}\Omega}{f^{p(m+j)}x^{1}}\mod p^{r_m}.\]
\end{prop}
\begin{proof}
See \cite[Lemma 1.10]{costa}. 
\end{proof}

\begin{rem}
The notations \(N_m\) and \(r_m\) will be explained in \ref{precision}.  
\end{rem}

\subsection{Precision}\label{precision} 
Since the coefficients of $F$ are in $\Q_p$, we need to compute the entries of $F$ to some sufficient finite precision such that one can \textit{provably} recover the correct integer coefficients of $Q(T)$ from a $p$-adic approximation of $F$. 

\begin{defi}
For $x, \overline{x} \in\Q_p$, we say an approximation $\overline{x}$ of $x$ has \textit{absolute p-adic precision} $N$ if $v_p(x-\overline{x})=N$. Suppose $x=p^{k}u$ with $k=v_p(x)$ and $u\in\Z_p^{\times}$. The \textit{relative $p$-adic precision} of the approximation $x+p^N\Z_p$ is 
$\rho(x;N)\coloneq \max\{0,N-v_p(x)\};$ equivalently, $\rho(x;N)$ is the largest natural number $r$ such that $u=p^{-k}x$ is determined modulo $p^r$, i.e. we know the first $r$ $p$-adic digits of $u$. 
\end{defi}

There are two aspects to the precision problem:
\begin{enumerate}
\item the $p$-adic precision needed in an approximation of $F$ to uniquely determine its characteristic polynomial, and 
\item the number of terms needed in the series expansion of the Frobenius action on differentials (c.f. \ref{frob}) in order to obtain a specific precision on the resulting Frobenius matrix.
\end{enumerate}
The first question is essentially answered by the Riemann hypothesis; see, e.g. \cite[Lemma 1.2.3 and \S4.3]{cyclicthreefold}. A helpful ingredient here is a classical result known as ``Newton over Hodge'' due to Mazur \cite{Mazur}, which relates the Hodge filtration to $p$-divisibility of the Frobenius matrix. In summary, Mazur's theorem says that 
$p^{n-m}|p^{-1}\sigma\left(\frac{x^{\beta}\Omega}{f^m}\right);$ also see \cite[\S3.3]{cyclicthreefold} for more details. Combined with \cite[Lemma 1.2.3]{cyclicthreefold}, one can deduce the \textit{relative} $p$-\textit{adic precision} $r_m$ that one needs to compute for $\sigma(x^{\beta}\Omega/f^m)$. We refer to $r_m$ as the \textit{relative precision} for the basis element $x^{\beta}\Omega/f^m$, and $r_m$ depends both on the pole order $m$ as well as the characteristic of the base field.


On the other hand, the analysis of how many terms in the series expansion we need in order to obtain the desired relative precision is a very delicate matter, and is one of the core difficulties in the study of explicit $p$-adic cohomological methods. 
The best known theoretical result in this direction is attributed to \cite[Proposition 3.4.6 and Proposition 3.4.9]{AKR}. 
We let $N_m$ denote the number of terms in the series expansion of $\sigma(x^{\beta}\Omega/f^m)$ in (\ref{frob}) that we need to compute $\sigma(x^{\beta}\Omega/f^m)$ to relative precision $r_m$; we refer to $N_m$ as the \textit{series order} of the basis element $x^{\beta}\Omega/f^m$, and $N_m$ depends on $m, r_m$, and $p$. 
\cite[Algorithm 3.4.10]{AKR} gives an algorithm for computing $N_m$; 
to the best knowledge of the authors, this algorithm has not been implemented.
We provide an implementation that calculates these precisions for an arbitrary projective
hypersurface.

Finally, for each $m=1,\ldots,n$, given $r_m$ and $N_m$, let $s_m \coloneq N_m+m-1$, we define the \textit{algorithm precision} to be 
$M\coloneq\max_{i=1,\ldots,n}\{r_m+v_p((ps_m-1)\!)-m+1\}.$ 
More specifically, this means that we are doing linear algebra over the ring $\Z_p/p^M \Z_p\cong\Z/p^M\Z$ in each of the step in the controlled reduction algorithm (c.f. \cite[1.5.1 Step 1]{costa}). The definition of $M$ will become apparent later. 



\section{Pole reduction}\label{sec:pole-reduction}
\subsection{Controlled Reduction}\label{subsec:cr}

To express $p^{-1}\sigma(x^{\beta}\Omega/f^m)$ as a $\Q_p$-linear combination of basis elements $B$ of $H^n_{\dR}(U_{\Q_p})$, we need to employ certain pole reduction algorithms to reduce the pole order of elements in the cohomology group using the Griffiths-Dwork relations (\ref{Griffiths-Dwork}). A straight-forward application of (\ref{Griffiths-Dwork}) is done in \cite{AKR}, but it can often turn a sparse form into a dense form. We now recall the method of \textit{controlled reduction}  \cite[\S 1.4]{costa}, a variation of \cite{AKR} that improves the time dependency on $p$ from that of $p^n$ in \cite{AKR} to $p^{1+\epsilon}$.  Controlled reduction is a family of algorithms,
which, generally speaking, works by applying the Griffiths-Dwork
relations (\ref{Griffiths-Dwork}) in a clever way that reduces the computation
from polynomial arithmetic to matrix-vector multiplication
and matrix addition\footnote{
    In theory, it reduces the problem to polynomial
    evaluation and matrix-vector multiplication,
    and one uses the fast-evaluation techniques of 
    \cite{BGS} to reduce the fast evaluation to matrix addition.
}.

Before we state the main result, we recall a technical definition. 

\begin{defi}\label{nondegenerate}
For $S\subseteq\{0,\ldots,n\}$, let 
$J_S\coloneq \langle \partial_i f\rangle_{i\in S} \oplus \langle x_i\partial_i f\rangle.$
We say that the hypersurface defined by the homogenous polynomial $f$ is $S$-\textit{smooth} if $J_S$ defines the empty scheme. We say $f$ is \textit{nondegenerate} if we can take $S=\emptyset$.
\end{defi}

\begin{rem}
When $S=\{0,\ldots,n\}$, $f$ is $S$-smooth is equivalent to $f$ is smooth. 
\end{rem}

\begin{prop}[Controlled Reduction]  \label{CR lemma}
Assume $f$ is $S$-smooth with $d\geq |S|$. Let $x^S\coloneq\prod_{i\in S}x_i$. Let $u=(u_i)\in \N^{n+1}$ and $v=(v_i)\in \N^{n+1}$ with $|v|=d$. Suppose further that for $i\in S$, if $v_i=0$, $u_i=0$. There is a $\Z_p$-linear map 
\[R_{u,v}:\Homog_{dn-n}\to\Homog_{dn-n}\]
such that for $g\in\Homog_{dn-n}$, 
\begin{equation}\label{Ruv}
m\frac{x^{u+v}g\Omega}{x^S f^{m+1}}\equiv x^u\frac{R_{u,v}(g)\Omega}{x^S f^m}\text{ in }H^n_{\dR}(U_{\Q_p}).
\end{equation}
Furthermore, setting $u=(x_0,\ldots,x_n)$, the linear map can be represented as a $\binom{dn}{n}\times\binom{dn}{n}$ matrix with entries being linear or constant polynomials in $\Z_p[x_0,\ldots,x_n]$. 
\end{prop}

\begin{proof}
See \cite[Proposition 1.15]{costa}. 
\end{proof}

We make explicit the definition of $R_{u,v}(g)$ and how to compute a matrix representation of $R_{u,v}$ following the proof given in \cite[Proposition 1.15]{costa}. Under the hypothesis in Proposition \ref{CR lemma}, we can find $g_i\in \Homog_{dn-n-|S|+1}$ such that 
\[\frac{x^vg}{x^S}=\sum_{i\in S}g_i\partial_if + \sum_{i\notin S}x_ig_i\partial_if.\]
After applying the Griffiths-Dwork relations and product rule, we have 
\begin{align*}
    m\frac{x^{u+v}g}{x^S}\frac{\Omega}{f^{m+1}}
    &\equiv x^u\left(\sum_{i\in S}\frac{u_ig_i+x_i\partial_ig_i}{x_i}+\sum_{i\notin S}(u_i+1)g_i+x_i\partial_ig_i\right)\frac{\Omega}{f^m}.
\end{align*}
Now let 
\[R_{u,v}(g)\coloneq x^S\left(\sum_{i\in S}\frac{u_ig_i+x_i\partial_ig_i}{x_i}+\sum_{i\notin S}(u_i+1)g_i+x_i\partial_ig_i\right),\] we then have \[m\frac{x^{u+v}g}{x^S}\frac{\Omega}{f^{m+1}}\equiv x^u\frac{R_{u,v}(g)}{x^S}\frac{\Omega}{f^m}.\]


\begin{rem}
We refer to the vector $v$ in $R_{u,v}$ as the \textit{direction of reduction}, c.f. Section \ref{sec:policies}. 
\end{rem}

\begin{rem}
The implementation \cite{pycontrolledreduction} only supports the case when $S=\{0,\ldots,n\}$, making it unusable for cubic threefolds and fourfolds. 
\end{rem}

For each fixed constant vector $v\in\N^{n+1}$, let $u=(u_0,\ldots,u_n)\in\N^{n+1}$. The $R_{u,v}$ linear map can be a priori represented by a matrix of linear polynomials in $\Z_p[x_0,\ldots,x_n]$; however, for better computational efficiency, we instead think of $R_{u,v}$ as being algorithmically represented by $n+2$ constant coefficient matrices $A_0,\ldots A_{n+1}$ with $R_{u,v}=\sum_{i=0}^{n}A_{i} u_i + A_{n+1}$. 

In order to apply Proposition \ref{CR lemma} to reduce the pole order of a monomial differential $x^{\beta}\Omega/f^m$ in practice, we need a consistent way of choosing $g\in\Homog_{dn-n}$ and some $u\in\N^{n+1}$ such that
\begin{equation}\label{u}
x^{\beta}=x^u g.
\end{equation}

In our implementation, as in Costa's implementation of \cite{costa}, the vector $u$ in (\ref{u}) is determined by a straight-forward algorithm as follows. 

\begin{algorithm}
\caption{Tweak$(I,m)$: remove $m$ units from the front of integer vector $I$}\label{algo:tweak}
\begin{algorithmic}[1]
\Require $I$, a vector of nonnegative integers; $m \ge 0$
\Ensure $I'$ after removing an integer vector $J$ with $|J| = \min(m, |I|)$ from the front
\State $count \gets 0$
\State $o \gets m$
\State $I' \gets \text{copy}(I)$
\While{$m > 0$}
    \For{$i \gets 1$ to $\mathrm{length}(I')$}
        \If{$I'[i] > 0$}
            \State $I'[i] \gets I'[i] - 1$
            \State $m \gets m - 1$
            \State \textbf{break}
        \EndIf
    \EndFor
    \If{$count > \mathrm{length}(I') \cdot o$}
        \State \Return $I'$
    \EndIf
    \State $count \gets count + 1$
\EndWhile
\State \Return $I'$
\end{algorithmic}
\end{algorithm}

The output of $\text{Tweak}(\beta, dn-n)$ gives the $u$ in (\ref{u}).

\subsection{Evaluation of $R_{u,v}$-matrices} 
\label{subsec:eval:ruv}
When we apply controlled reduction for the purposing of reducing the pole order of differentials via the $R_{u,v}$-matrices in practice, we write $u=(x_0,\ldots,x_n)+yv$ for some fixed constant vectors $(x_0,\ldots,x_n),v\in\N^{n+1}$, and variable $y\in\N$. Let $M(y)\coloneq R_{(x_0,\ldots,x_n)+yv,v}=Ay+B$, where $A$ and $B$ are constant coefficient matrices. To reduce the pole order of a differential $\omega$ by $k$, we need to evaluate the expression 
\begin{equation}\label{fasteval}
M(0)\cdots M(k)w,
\end{equation}
where $w$ is the vector of coefficients of some $g$ as in (\ref{Ruv}). The question of how one chooses $v$ will be addressed in Section \ref{sec:policies}. For now, we focus on how to efficiently evaluate (\ref{fasteval}). A priori, (\ref{fasteval}) can be evaluated via $k$ matrix multiplications of matrices of size $\binom{dn}{n}\times\binom{dn}{n}$ and a matrix-vector multiplication. Instead, (\ref{fasteval}) can be evaluated via $O(k)$ matrix-vector multiplications and $O(k)$ matrix additions as follows. 

\begin{algorithm}
\caption{Fast Evaluation of (\ref{fasteval})}\label{horner}
\begin{algorithmic}[1]
 \State $M \gets kA+B$
 
  \For{$i = k$ \textbf{downto} $0$}
   \State $w \gets Mw$
   \State $M \gets M-A$
  \EndFor
  \State \Return $w$
\end{algorithmic}
\end{algorithm}

We introduce a customized data structure for an efficient implementation of Algorithm \ref{horner} in Section \ref{PEP}.

\subsection{Outline of the pole reduction process} 
For each element in a monomial basis $B$ of $H^n_{\dR}(U_{\Q_p})$ as described in Section \ref{basis}, recall that our objective is to express (\ref{frob}) as a $\Z_p$-linear combination of elements of $B$ by means of a pole reduction procedure (elements in $B$ have pole order at most $n$). There are two main steps for how this is done in controlled reduction: 
\begin{enumerate}
\item [(0)] Apply Algorithm \ref{algo:tweak} to express the terms in the RHS of (\ref{frob}) in the format of the LHS of (\ref{Ruv}). 
\item [(1)] Iteratively apply the $R_{u,v}$-matrices to reduce the pole order of the terms in (\ref{frob}) to exactly $n$. 
\item [(2)] Directly apply the Griffiths-Dwork relation (\ref{Griffiths-Dwork}) as in \cite{AKR} to express the output from Step 1 as a linear combination of elements of $B$.
\end{enumerate}
Most of the computational complexities lies in Step 2, which we repackage as an abstract optimization problem in Section \ref{subsec:abstract} and propose various solutions in Section \ref{sec:policies}.

\subsection{The Abstract Reduction Problem}\label{subsec:abstract}

We define an optimization problem that models the problem of controlled reduction.
Recall our convention from earlier that the natural numbers include zero.

Let \(d, N \in \mathbb{N}\). We call \(d\) the \textit{degree} and \(N\) the \textit{target norm}. 
Let \(|\cdot|\) denote the (restriction of the) $\ell_1$ \textit{norm} on \(\mathbb{N}^{n}\), i.e., 
 for \(u = (u_{0}, \ldots, u_{n}) \in \mathbb{N}^{n}, |u| = u_{0} + \ldots + u_{n}\).

At the start of the problem, we are given
\(U\), a finite list whose elements lie in \(\mathbb{N}^{n}\)
with the following property:
for every \(u \in U\), 
\(N = |u| - md\) for some \(m \in \mathbb{N}\).
We are also given, for each \(u\in U\) and \(k\in\N\), 
a finite set \(V(u,k) \ins \{v \in \mathbb{N}^{n+1} | |v| = d\}\).
We call \(V(u,k)\) the set of \textit{reduction direction choices}.

\newcommand{\all}{\operatorname{all}}
\newcommand{\full}{\operatorname{full}}

The following are the basic examples of \(V(u,k)\);
others are variants of these.

\begin{exam}
    Let \(V_{\all}(u,k)\) be the set \(\{v \in \mathbb{N}^{n+1} \colon |v| = d, u - kv \in
    \mathbb{N}^{n+1}\}\).
    That is, the set of all elements of $\ell_1$ norm $d$,
    with the condition that the entries remain nonnegative after subtracting by \(v\) $k$ times.
\end{exam}

\begin{exam}
    Let \(V_{\full}(u,k)\) 
    be the subset of \(V_{\all}(u,k)\) with the additional condition
    that if \(u_{i} \neq 0\), then \(v_{i} \neq 0\).
\end{exam}

\begin{exam}
    Let \(S \ins \{0, \ldots, n\}\).
    Let \(V_{S}(u,k)\)
    bet the subset of \(V_{\all}(u,k)\) with the additional 
    condition that if \(i \in S\), then 
    if \(u_{i} = 0\), \(v_{i} = 0\).
\end{exam}

The process of controlled reduction can be formulated as a 1-player game. 
We score the game via three cost constants \(M, D\), and \(S\),
for which we assume that \(D \ll S \ll M\).
The interpretation of the constants is as follows:
\(M\) denotes the cost of reducing one term by one pole order.
This is dominated, in practice, by the cost of a single
matrix-vector multiplication (there is also a matrix subtraction,
but in practice the memory bottleneck of matrix-vector multiplication
takes a lot more time, even though technically both have
\(O(s^{2})\) operations, where \(s\) is the side length of the matrices).
\(S\) is the startup cost to compute the \(A\) and \(B\) matrices
as described in Section \ref{subsec:eval:ruv}.
This consists of matrix additions and scalar multiplications, which 
like the subtraction are much faster than a matrix-vector multiplication.
Finally, \(D\) is the cost of adding two vectors which correspond to duplicates from \(U\), which 
in theory is a single vector addition, but in practice again memory
overhead dominates.

With that in mind, we may formulate our game as follows:
On each turn, the player does one of two moves, where each move accumulates a cost, and the problem
is to minimize the cost.


\begin{enumerate}
    \item Choose \(u \in U\), \(k \in \mathbb{N}\), and
        \(v \in V(u,k)\).
        Replace \(u\in U\) by \(u - kv\) at a cost of \(Mk + S\); or 
    \item Remove all duplicates in \(U\) at a cost of \(D\) per 
        duplicate.
        This operation may not be done
        twice in a row.
\end{enumerate}

We refer the first move as a reduction of chunk size $k$. 
Throughout the game, all \(u\) must have \(N \leq |u|\).
The game ends when all \(u \in U\) satisfies $|u|=N$.
Note that any game is guaranteed to finish;
the goal is to finish the game with minimal cost.
One must balance the savings of taking the first
move with large \(k\) (thus avoiding the startup cost \(S\)) 
with the savings of ``combining like
terms'' in the second move.

In general, we fix the reduction direction choices \(V(u,k)\) once and for all,
and expect to solve the problem for many choices of \(U\).
Given a fixed choice of \(V(u,k)\) (for example, \(V_{\all}(u,k)\)),
a \textit{reduction policy} is a program that
plays the reduction strategy game.
The optimization problem is to find a good policy that runs with 
(close to) minimal cost for the \(U\) that appear in practice. We discuss various versions of reduction policies in Section \ref{sec:policies}. 

\begin{rem}
When doing controlled reduction, we always take the \textit{degree} $d$ to be the degree of the defining equation of the hypersurface $X$, and the \textit{target norm} $N$ to be $n$, the dimension of the ambient projective space of $X$. The input $U$ is obtained via Algorithm \ref{algo:tweak}. 
\end{rem}

\section{Reduction Policies}\label{sec:policies}

We now describe three practical implementations of reduction
policies. 
Following the tradition of previous implementations like
\cite{costa} and \cite{ToricControlledReduction},
all of our policies begin with the following steps: 
\begin{enumerate}
\item Choose a $u\in U$ to reduce.
\item Choose $v\in V(u, 1)$. 
\item Choose $k\in\N$ such that $v$ remains in $V(u,k)$. 
\end{enumerate}
The \(v\) is always chosen using Algorithm \ref{alg:chooseV:greedy}, a greedy algorithm, which always outputs $v\in V(u,1)$. Recall that for all \(v \in V(u,1)\), \(|v| = d\),
by the setup of Section \ref{subsec:abstract}.


\begin{algorithm}
\caption{Greedy algorithm to choose \(v\)}\label{alg:chooseV:greedy}
 \begin{algorithmic}[1]
\Require Inputs: $d\in\N$, $u\in\N^{n+1}$, $S\subseteq\{0,\ldots,n\}$
\Ensure Output: $v\in V(u, 1)$
  \State $v \gets [0, \ldots, 0]$
  \State \(i \gets d\)
  \For{\(s \in S\)}
    \If{\(u[s] \neq 0\)}
      \State \(v[s] \gets 1\)
      \State \(i \gets i - 1\)
      \If{\(i = 0\)}
        \State \textbf{break}
      \EndIf
    \EndIf
  \EndFor

  \While{$i > 0$}
    \For{\(m \in 0\)  \textbf{upto} \(n\)}
      \If{\(v[m] < u[m]\)}
        \State \(v[m] = v[m] + 1\)
        \State \(i = i - 1\)
      \EndIf
    \EndFor
  \EndWhile
  \State \Return $w$
\end{algorithmic}
\end{algorithm}

Observe that Algorithm \ref{alg:chooseV:greedy} guarantees that $v$ satisfies the assumption
of Proposition \ref{CR lemma}.

\begin{rem}
The reduction policy structure of choosing a \(v \in V(u,1)\) given
each \(u\) is not required and possibly not optimal.
It would be interesting to explore this in future work.
\end{rem}

\subsection{$p$-chunk reduction}\label{costachunks}

Our first reduction policy, which we call 
the \textit{\(p\)-chunk reduction}, is the one implemented
in \cite{costa} and \cite{ToricControlledReduction}.
It takes advantage of the fact that for all \(u, u^{\prime} \in U\), 
we have \(|u| \equiv |u^{\prime}| \mod p\),
which follows from Equation \ref{frob}.
Let $L\coloneq \max\{|u|: u \in U\}$ and 
 \(U_{\max} = \{u \in U \mid |u| = L \}\).
The \(p\)-chunk reduction policy goes as follows:

\begin{enumerate}[(1)]
    \item For every \(u \in U_{\max}\), 
        choose a \(v\) using Algorithm \ref{alg:chooseV:greedy}
        to reduce \(u\) by a chunk of size \(p\).
        If \(|u| = p\),
        then reduce by a chunk of size \(p - n\).
    \item Combine like terms.
    \item Recalculate \(L\) and \(U_{\max}\).
    \item Repeat steps 1-3 until the game described in Section \ref{subsec:abstract} is won.
\end{enumerate}

Towards the end of the reduction, many like terms will be combined
as there are fewer possibilities for \(u\) when it has smaller norm.

The \(p\)-chunk reduction policy is most effective when \(p\) is very large,
since if \(p\) is small the start-up costs will accumulate and make the
algorithm slow.

\subsection{Depth-first reduction}\label{subsec:naive reduction}

Our second reduction policy, which we call \textit{depth-first reduction},
works
as follows. For each $u\in U$, 
\begin{enumerate}
\item Choose a \(v\) by Algorithm \ref{alg:chooseV:greedy}.
\item Find the largest \(k\) such that \(v\) is still in \(V_{\full}(u,k)\). Then reduce \(u\) by a chunk of size \(k\).
\item Repeat steps 1-2 until the game described in \S\ref{subsec:abstract} is won. 
\end{enumerate}
Unlike $p$-chunk reduction, this strategy never combines like terms.

Since the depth-first policy prioritizes
long reduction chunks over combining like terms,
it works best when the size of \(U\) is small.
For example, at smaller primes,
Equation \ref{frob} will have fewer cross terms.
Also, if the equation \(f\) of the hypersurface
is sparse, then Equation \ref{frob} will have
less terms and the performance of 
depth-first reduction will be improved.
This policy outperforms \(p\)-chunks reduction
in our situations, for example in the case
of quartic K3 surfaces of medium-low characteristic. 

\subsection{Variable-by-variable reduction}\label{subsec:varbyvar}

The third reduction policy that we have implemented, 
\textit{variable-by-variable reduction},
aims to balance the approaches of the previous two.
Here we assume that \(X\) is $\{n\}$-smooth in the sense of Definition \ref{nondegenerate}.
In this case, the first
for loop of Algorithm \ref{alg:chooseV:greedy}
does not affect its output.
The policy works by 
\begin{enumerate}[(1)]
    \item For each \(u\), choose a \(v\) by Algorithm \ref{alg:chooseV:greedy}.
    \item Choose \(k\) to be the largest \(k\) such
        that \(v \in V(u,k)\). 
        Reduce \(u\) by a chunk of size \(k\).
    \item Repeat steps 1-2 until the \(n\)-th component \(u_{n}\)
        of \(u\) is zero.
    \item Now that all \(u\) have \(u_{n} = 0\), combine like terms
    \item Repeat steps 1-4 for \(u_{n-1}\), 
        \(\ldots u_{0}\) until the game ends.
\end{enumerate}

This policy works better when \(U\) is more dense, so
there are likely to be many \(u \in U\) with the same first \(k\) entries.
It also works better when the degree of the equation \(f\) is low,
note that for the \(k\)-th entry, step (3) happens at most \(d\) 
times for each \(u\).
It has the added benefit of being well suited
to the GPU--it requires only a few matrices
to be stored on the GPU at a time.

\subsection{Outlook}

The authors do not suspect that our reduction policies are the most optimal.
In particular, one thing that all of our reduction policies
have in common is that they do not inspect \(U\) at all.
We hope that the framework of the abstract reduction problem
inspires future work on more efficient reduction policies,
which may lead to significant speedups.
We hope to address this in future work.

\section{Algorithms and Data Structures}\label{sec:algo}

A fast implementation of controlled reduction relies on efficient 
implementations of many auxiliary algorithms that are not, a priori, related to controlled reduction.
In particular, to implement controlled reduction, one must at least have
(1) a matrix multiplication 
mod \(p^{n}\), and (2) fast evaluation of linear recurrences of polynomials (with coefficients that
are matrices with entries
mod \(p^{n}\)).
Furthermore, fast implementations of these algorithms are not always widely available or
immediately
usable, though if we are on the CPU, the FLINT library has fast implementations
of matrix multiplication.
We provide implementations of two variants of matrix multiplication,
and we introduce data structures for managing linear recurrences.
Our implementations are released as open-source software.




\subsection{Partially Evaluated (linear) Polynomials--a novel data structure for fast evaluation of
matrices of polynomials}\label{PEP}

We introduce 
an abstract data structure that can be used to encapsulate the
evaluation procedure described in Section \ref{subsec:eval:ruv}.
Let \(A\) be a (possibly noncommutative) ring.
For us, \(A\) is always \(M_{n}(B)\) for 
some commutative ring \(B\).
We consider an element 
\(R \in A[u_{0}, \ldots, u_{n}, v_{0}, \ldots, v_{n}]\)
of degree \(1\).
We let \(V \in \mathbb{N}^{n+1}\) be
some finite set of values.
For the evaluation in Section \ref{subsec:eval:ruv},
we need to store polynomials 
\(R^{\prime} = R(v_0, \ldots, v_{n}) \in A[u_{0}, \ldots, u_{n}]\) 
for each value \(v\in V\) with $|v|=d$.
Creating \(R\) can be an expensive operation,
while storing an \(R(v_{0}, \ldots, v_{n})\) can take a lot of memory.
Finally, we must manage the process of evaluating 
the matrix
\(R^{\prime}(x_{0} + yv_{0}, \ldots, x_{n} + yv_{n})\)
to get the matrices \(A\) and \(B\) from Section \ref{subsec:eval:ruv};
this also is potentially an expensive operation.


A \textbf{Partially Evaluated Polynomial} (PEP) is
a key-value data
structure which stores evaluations of \(R\) for
various fixed values \(v \in V\),
manages the storage of the \(R^{\prime}\),
and contains methods for evaluation of linear
recurrences as in \ref{subsec:eval:ruv}.
A PEP data structure has
has three main operations:

\begin{itemize}
    \item \texttt{compute}, a function
        provided by the user 
        which inputs a \(v_{0}\) and returns
        a list of the coefficients of \(R(u,v_{0}) \in A[u_{0}, \ldots, u_{n}]\)
    \item \texttt{get}, which gets \(R^{\prime} = R(u,v)\).
    \item \texttt{eval\_to\_linear!}, which takes the output
        of \texttt{get} and an \(x = (x_{0}, \ldots, x_{n})\) 
        and evaluates \(R(x + yv, v) \in A[y]\)
        as in Section \ref{subsec:eval:ruv}.
        This can then be used for fast evaluation
        as in Algorithm \ref{horner}.
\end{itemize}

PEP structures generally have some sort of underlying
dictionary and are distinguished by how they 
manage the creation of key-value pairs.
We provide an abstract Julia type, \texttt{AbstractPEP},
as a supertype for different PEP implementations.
We provide five such implementations:

\begin{itemize}
    \item \texttt{EagerPEP}, which is backed by a dictionary
        and \texttt{compute}s
        all possible keys upon initialization
        (optionally, in parallel).
    \item \texttt{LazyPEP}, which is backed by a dictionary 
        and \texttt{compute}s
        keys lazily.
    \item \texttt{LRULazyPEP}, which is backed by an LRU Cache of a fixed size,
        and computes keys lazily.
    \item \texttt{LRUCachePEP}, which
        assumes the output coefficients are expected
        to be in GPU memory.
        It is backed by an internal PEP
        on the CPU, and
        an LRU cache whose entries are in GPU memory,
        and it manages moving entries to and from the GPU.
    \item \texttt{LFUDACachePEP},
        which gives the same behavior as \texttt{LRUCachePEP}
        but uses an LFUDA cache.
\end{itemize}

Julia's JIT-compiler and dynamic dispatch make PEP objects
more or less interchangeable, allowing one to rapidly test 
different memory management strategies.
The optimal choice of PEP object tends to depend heavily on
the choice of reduction policy; for example, we anecdotally observed that
the LFUDA backing had fewr cache misses with the depth-first
policy (Section \ref{subsec:naive reduction}) than LRU,
while the LRU backing had fewer misses with the
variable-by-variable policy (Section \ref{subsec:varbyvar}).
As another anecdote, the \texttt{LRULazyPEP} became useful for our
computations with cubic fourfolds, when our CPU memory became a 
constraint (in addition to GPU memory, which was also a constraint).

\subsection{\texttt{CuModMatrix}: mod \texorpdfstring{\(N\)}{p power of n} matrix multiplication
on GPUs}

We provide Julia a (dense) GPU matrix datatype for mod \(N\) linear 
algebra, \texttt{CuModMatrix}, implemented using CUDA.jl. 
\texttt{CuModMatrix} implements standard matrix operations 
(matrix addition and multiplication, inverse,
LU decomposition, etc.).
Each \texttt{CuModMatrix} has an element type which
must be a Julia integer of float type (e.g. \texttt{Int8}
or \texttt{Float64}) that backs the elements
of the matrix.
If the element type is a floating point type,
the mantissa is used to efficiently obtain 
a smaller integer type
({e.g.} \texttt{Float64} becomes \texttt{Int53},
\texttt{Float32} becomes \texttt{Int24}).
We prefer these floating-point-backed types,
as we can use Nvidia's CUBLAS \cite{nvidia-2024-cublas}.

We use the method of 
\cite[Algorithm~1]{bglm-2023-gpu-matmul-modp}
for mod \(N\) matrix multiplication.
This involves breaking the matrix into vertical stripes,
whose lengths are determined by the number of possible operations
without overflow. 
The algorithm uses CUBLAS to multiply the stripes and
reduces the entries modulo \(N\).

\subsection{\texttt{KaratsubaMatrix}: Mod \(p^n\) matrix multiplication on GPUs via Karatsuba
multiplication}


Often, the necessary \(p\)-adic precision for controlled
reduction is too big to do matrix multiplications
within the \texttt{Float64}
datatype (i.e. \texttt{Int53}) which is the largest
one available for use with cuBLAS. For cubic threefolds over $\FF_{11}$, for example,
the precision analysis
of \cite{AKR}
indicates that we do matrix multiplications \(\mod 11^8 \).
However, 
we have that
\(\left\lfloor 
\frac{2^{53}}{(11^8 - 1)^{2}} \right\rfloor
\approx \left\lfloor 0.19 \right\rfloor = 0 \),
meaning that we cannot justify 
even a single multiplication.

To overcome this, we use Karatsuba multiplication.
To multiply two matrices 
\(A,B \mod N\), we write 
 \(A = A_1 + MA_2\) and \(B = B_2 + MB_2\), where we require $N|M^2$. 
Then 
\[AB = A_{1}B_{1} + M(A_{2}B_{1} + A_{1}B_{2})
+ M^2A_{2}B_{2}.\]
We can re-express the term $A_2B_1 + A_1B_2$ as the difference between
products
\((A_1 + A_2)(B_1 + B_2) - A_{1}B_{1} - A_{2}B_{2}\),
so computing the multiplication $AB$ now requires only 3, instead of 4, matrix 
multiplications.
We may want to choose \(M\) 
to be \(11^{4}\) in the above example, in which case we can discard
the last term $M^2 A_2B_2$.
This approach is considered in \cite{bao-2024-thesis}.

We provide a GPU-accelerated \texttt{KaratsubaMatrix} type that also
supports 
basic operations such
as addition, subtraction, negation, and scalar multiplication.



\section{Examples}\label{sec:newton-realization}
In this section, we survey prior work in the literature and discuss the methodologies adopted for finding all of our explicit examples.
The equations for each example can be found in \cite{zenodo}.

\subsection{Methodology}
There are two principal methods for finding explicit examples of hypersurfaces
with interesting Newton polygons. 

\subsubsection{Method 1: Fully random guessing}
We randomly choose equations of hypersurfaces,
computing the zeta function of each, hoping
to find one with the desired Newton polygon.
The idea is to exploit the fact that Newton strata jump in codimension 1 
\cite{de-jong-oort-2000-newton-strata},
so we expect that by randomly sampling
the closure of some stratum, finding a hypersurface in a stratum one jump higher 
has probability roughly \(1 / p\).
Thus, if a strata has codimension \(n\), the chance of finding
such a hypersurface has probability \(1 / p^{n}\).
We can formally justify these expectations using
the Lang-Weil estimate \cite{lang-weil-1954-estimate},
though that bound is not so tight as to determine what happens practically.
This method is more effective the smaller \(p\) is.


\subsubsection{Method 2: Random deformations of a fixed example}

In this method, we fix a hypersurface and add random monomials to it. 
We try to fix a hypersurface which has a Newton polygon as high as possible,
ideally supersingular.
For example, we may choose the Fermat hypersurface or a member of a
Dwork-type pencil.
The idea is to exploit a phenomenological observation: equations with higher Newton polygons are
often similar-looking to other equations with higher Newton polygons.
This phenomenon has not been fully explored in the literature, and in fact we hope to do so in
future work.
However, there are theoretical reasons to expect such similarities.
Since the Newton strata are locally closed strata in the moduli space,
they are defined by polynomial equations.
If the defining equations of these strata have low degree, that could lead to
``similar-looking equations'' among various points of the strata.
For example, if we consider quartic K3 surfaces, we may compute the equation of
the \(h = 2\) Newton stratum using Fedder's criterion \cite[\S2]{ma-polstra-2021-F-sing-comm-alg}, which
has degree \(p-1\).


\subsection{Curves}
Let $g\geq 1$, $\mathcal{M}_g$ be the moduli space of genus $g$ curves, and $\mathcal{A}_g$ be the moduli space of $g$-dimensional
principally polarized abelian varieties. One is generally interested in how the \textit{open Torelli locus}, i.e. the image of 
$\mathcal{M}_g$ under the Torelli morphism $\tau_g:\mathcal{M}_g\to \mathcal{A}_g$ sending a curve to its Jacobian, 
intersects various strata of $\mathcal{A}_g$, in particular the $p$-rank strata and Newton polygon strata. 
For the former, 
Achter and Pries \cite{Achter-Pries-monodromy} has shown that there exists a curve $C/\overline{\FF}_p$ of genus $g$ and $p$-rank $f$ for all $0\leq f\leq g$; the same result holds for hyperelliptic curves as well \cite{Achter-Pries-hyperelliptic}.  
We are interested in whether the result of Achter and Pries holds over $\FF_p$ for plane curves.
We study quintics, sextics, and heptics, and use a combination of Method 1 and Method 2. We find

\begin{enumerate}
    \item Plane quintics of all possible \(p\)-ranks for \(p = 7, 11, 13\). See \cite{zenodo}.
    \item Plane sextics of all \(p\)-ranks \(\geq 6\) for p = \(7,11,13\). See \cite{zenodo}.
    \item Plane heptics which are ordinary and almost ordinary for \\\(p = 7,11,13,17,19,23,31\). See \cite{zenodo}.
\end{enumerate}

\subsection{Quartic K3 surfaces}

Let \(X\) be a K3 surface; the Hodge numbers of \(H^{2}_{rig}(X)\) are 
\(1, 20, 1\).
Thus, the only possibilities for the Newton polygon are 
\begin{enumerate}[(1)]
    \item A straight line of slope 1 for 22 units,
        i.e. the supersingular case
    \item Lines connecting the vertices
        \((0,0), (h,h-1), (22-h,22-h-1), (22,22)\)
        for some \(h \in \{1, \ldots, 10\}\).
\end{enumerate}

The \(h\) in possibility (2) is called the
\textit{Artin-mazur height} of \(X\).
If \(X\) is supersingular, we say that \(h = \infty\).
Alternatively, note that the first slope of the Newton polygon is
\(\frac{h-1}{h} = 1 - \frac{1}{h}\). 

Over an algebraically closed field \(k\) of characteristic 
\(p\), K3 surfaces of all heights exist by 
\cite{taelman-2016-k3-given-l-function}.
Quartic K3 surfaces of any height
exist over $\FF_2$
by \cite{K3F2},
over \(\mathbb{F}_{3}\) by 
\cite{kty-2022-fedder},
and over \(\mathbb{F}_{5}\) and \(\mathbb{F}_{7}\) 
by \cite{bgp-2025-k3-surfaces}.
Over \(\mathbb{F}_{11}\) and \(\mathbb{F}_{13}\),
\cite{bgp-2025-k3-surfaces} provides examples
of heights less than or equal to five.
We obtain examples for quartic K3 surfaces of heights
up to and including 4
over \(\mathbb{F}_{p}\) for \(p = 17, 19, 23\).
See \cite{zenodo}.

Here we use Method 1 to obtain our result;
we use Fedder's criterion \cite[\S2]{ma-polstra-2021-F-sing-comm-alg} to rule
out surfaces of height 1, so we have to make fewer random guesses to find surfaces of higher height.

\subsection{Quintic Surfaces}

Let \(X\) be any algebraic variety over a perfect
field \(k\) of characteristic \(p\).
Then \(X\) is \textit{Hodge-Witt} if the 
cohomology groups
\(H^{j}(X, W\Omega_{X}^{i})\)
are finitely generated \(W\)-modules for all
\(i\) and \(j\)
(where \(W\Omega_{X}^{i}\) are the sheaves
appearing in the de-Rham Witt complex).
More generally, Illusie and Raynaud \cite{Illusie-Raynaud}
introduce the notion of the \textit{domino numbers}
\(T^{i,j}\) associated to a variety in 
characteristic \(p\).
If \(X\) is a projective surface, then by 
\cite[\S2.23, \S2.37]{Joshi},
the only interesting domino for \(X\) is 
\(T^{0,2}\); all others are either zero or
related to this one by a duality theorem.
Furthermore, we have \(0 \leq T^{0,2} \leq \dim H^{n}(X, \mathcal{O}_{X})\),
and \(T^{0,2}\) is given as a formula depending on 
the Hodge numbers and slopes of the Newton polygon.
We may visually understand this formula as follows:
if we take the vertical distance between the Newton polygon and Hodge
polygon, the domino number \(T^{0,2}\) is the minimum such difference in the
slope 1 part of the Hodge polygon.
If \(X\) is a quintic surface, then \(\dim H^{n}(X, \mathcal{O}_{X}) = 4\) and so $T^{0,2}=0,1,2,3,4$; $X$ is Hodge-Witt if and only if $T^{0,2}=0$.

In characteristics \(7, 13\), and \(17\), the Fermat quintic
is supersingular. 
By applying Method 2 to the Fermat quintic, we obtain
examples of quintic surfaces of all possible domino numbers
in these characteristics.
However, in characteristic \(11\), the Fermat quintic
is ordinary, and the authors do not know of a supersingular
example.
Moreover, given the large number of possible Newton polygons
and the large amount of time that computing the zeta function of a single dense example takes
(about one hour), Method 1 does not seem viable
for quintic surfaces.

\subsection{Cubic Threefolds}
The Hodge numbers of a cubic threefold are 
\(0,5,5,0\).
Thus, their possible Newton polygons are the same as 
curves of genus 5, with all slopes increased by one.
The authors are not aware of any results on the
realization of Newton strata of cubic threefolds,
even over \(\overline{\mathbb{F}}_p\), in the literature.

We find cubic threefolds in characteristic \(p = 7,11\) and \(13\) 
which are ordinary and almost ordinary,
using Method 1.
See \cite{zenodo}.

\subsection{Cubic Fourfolds}
\label{subsec:fourfolds}

Let \(X\) be a cubic fourfold.
The Hodge numbers of \(H^{4}_{rig}(X)\) are
\(0,1,21,1,0\); hence the possibilities for the Newton polygon of $X$
are the same as for a K3 surface, 
except that the slopes are all increased by one.
Thus, primitive cohomology ``looks like'' the 
cohomology of a K3 surface that has been
Tate twisted by 1.
As such, we may define the \textit{Newton height} of $X$
analogously as \(h = \frac{1}{2 - \newt_{1}}\),
or \(2 - \frac{1}{h} = \newt_{1}\), where $\newt_1$ is the slope of the first segment in the Newton polygon of $X$.

There is some existing literature on examples of cubic fourfolds. \cite{cubic-fourfolds-f2} computes all cubic
fourfolds over \(\mathbb{F}_{2}\) up to isomorphism, and verify that all possible Newton polygons
occur, which fully resolves Question \ref{quest:strata:realization} over \(\mathbb{F}_{2}\).
Furthermore, Bragg and Perry, in an unpublished work \cite{bragg-perry}, calculate the Newton height of the
Dwork pencil of cubic fourfolds 
\(
x_0^3 + x_1^3 + x_2^3 + x_3^3 + x_4^3 + x_5^3 + \lambda \bigl(x_0 x_1 x_2 + x_3 x_4 x_5\bigr)
.\)
They determine that for all \(p\), there exists a unique \(\lambda_0\) for which the corresponding fourfold
\(X_{\lambda_0}\) is supersingular, and \(X_{\lambda}\) is ordinary for all other values of \(\lambda\). Then they deduce by a deformation argument that there exist cubic fourfolds of all heights over \(k = \overline{\mathbb{F}}_p\). However, this argument is nonconstructive, and
so it only answers Question \ref{quest:strata:realization} for the ordinary case.


Using Method 1, we find a cubic fourfold of heights \(1\) and \(2\) in characteristic \(p = 7\). 
See \cite{zenodo}.


\section{Performance Comparisons}\label{sec:performance}
In this section, we compare the practical performance of our reduction policies to existing implementations in various settings. Although our ultimate goal is to compute zeta functions of higher dimensional hypersurfaces, the method of controlled reduction \textit{can} be used to compute the zeta function of smooth plane curves. We provide some performance comparisons for completeness sake. For smooth plane quintics (curves of genus 6), we compare the performance of our implementation of the depth-first reduction policy \ref{subsec:naive reduction} together with fast evaluation (c.f. Algorithm \ref{horner}) on a single CPU against Tuitman's algorithm \cite{Tuitman}, a Kedlaya-type algorithm, as well as Kyng's algorithm \cite{Kyng}, an implementation of Harvey's trace formula method \cite{Harvey-Trace-Formula}, both of which are available in \textsc{Magma} \cite{Magma}. The running times for each of the aforementioned algorithms on smooth plane quintics defined by random dense polynomials $f\in\FF_p[x_0,x_1,x_2]_5$ can be found in Table \ref{table:plane quintics}. We note that Kyng's algorithm is considered to be the state of the art for computing the zeta functions of higher genus (possibly singular) plane curves. We did not expect to obtain better performance than either of Tuitman's or Kyng's algorithm. We observe in Table \ref{table:plane quintics} that our algorithm has a similar performance as Tuitman's algorithm. 

\begin{table}
\caption{Algorithms for plane quintics over $\FF_p$. Times in one Apple M3 Pro CPU core-seconds.}
\label{table:plane quintics}
\begin{tabular}{cccc}
$p$ & \ref{subsec:naive reduction} & \cite{Tuitman} & \cite{Kyng} \\
\hline 
11 & 0.911 & 8.24& 1.32 \\
13 & 0.683 & 1.3 & 1.26 \\
17 & 0.899 & 2.21 & 1.38 \\
19 & 0.987 & 2.36 & 1.43 \\
23 & 1.157 & 2.66 & 1.45 \\
29 & 3.898 & 3.13 & 0.57 \\
31 & 1.488 & 3.37 & 0.65 \\
37 & 1.692 & 18.72 & 0.76 \\
41 & 4.679 & 4.22 & 0.72
\end{tabular}

\end{table}

\cite{ToricControlledReduction} is the current state the art for computing zeta functions of
projective hypersurfaces beyond curves. In Table \ref{table:sparse-surfaces}, we see that our
implementation using the depth-first reduction policy \ref{subsec:varbyvar} obtains performance
surpassing \cite{ToricControlledReduction} on quartic surfaces with six terms. We note that we use
multithreading when computing the zeta functions of these surfaces in our experiments in practice,
which speeds up the computation even more; multithreading is well-supported in our implementation.
\cite{ToricControlledReduction} cannot compute zeta functions of random quintic surfaces due to
memory constraints, whereas each random quintic surface usually takes about one Apple M3 Max 
core hour to finish using our implementation with multithreading. The implementation in
\cite{ToricControlledReduction} can only handle cubic threefolds and fourfolds over $\FF_p$ for
$p\geq 13$, and it runs out of memory on all machines available to us (under a 32GB limit) when
computing the zeta function of random cubic fourfolds. Our implementation using the
variable-by-variable reduction policy \ref{subsec:varbyvar} takes between 2-3 hours on 
an Intel i5-8400 
with one Nvidia RTX 3070 GPU for random cubic fourfolds over $\FF_7$, and between 15-20
minutes (depending on the PEP structured used \ref{PEP}) for the Fermat cubic fourfold over $\FF_7$.

\begin{table}[ht]
\caption{Algorithms for degree $d$ surfaces in $\PP^3$ over $\FF_p$ with six terms. Times in one AMD Ryzen 5 2600 core-seconds.}
\label{table:sparse-surfaces}
\centering
\small
\setlength{\tabcolsep}{6pt}
\renewcommand{\arraystretch}{1.05}

\begin{tabular}{r rrr @{\hskip 18pt} rrr @{\hskip 18pt} rrr}
& \multicolumn{2}{c}{Cubic Surfaces}
& \multicolumn{2}{c}{Quartic Surfaces} \\
& \multicolumn{2}{c}{$(n=3, d=3)$}
& \multicolumn{2}{c}{$(n=3, d=4)$} \\
\cmidrule(lr){2-3}\cmidrule(lr){4-5} \\
$p$
& \ref{subsec:varbyvar}& \cite{ToricControlledReduction} 
&  \ref{subsec:varbyvar}& \cite{ToricControlledReduction} \\
\midrule
7 & 0.15 & - & 7.22 & - \\
11 & 0.13 & 0.13 & 9.71 & 12.79 \\
13 & 0.08 & 0.07 & 7.85 & 10.70 \\
17 & 0.09 & 0.07 & 10.29 & 12.98 \\
19 & 0.09 & 0.07 & 11.11 & 13.94 \\
23 & 0.09 & 0.07 & 4.69 & 7.14 \\
29 & 0.10 & 0.07 & 5.28 & 6.93 \\
31 & 0.10 & 0.07 & 5.56 & 7.34 \\
\bottomrule
\end{tabular}
\end{table}

\bibliographystyle{amsalpha}
\bibliography{source.bib}

\appendix

\end{document}